\documentclass[a4paper,11pt]{amsart}

\usepackage{amsmath,amsfonts,amssymb,amscd,amsthm,amsbsy,upref}
\usepackage{verbatim}
\usepackage{amssymb,amsthm,amsmath,enumerate}
\usepackage{amsrefs}

\numberwithin{equation}{section}

\usepackage{enumerate}

\usepackage{amsrefs}
\usepackage{amsthm}
\usepackage{amsmath}		
\usepackage{amssymb}	
\usepackage{amsfonts}

\usepackage{tikz}
\usepgflibrary{shadings}
\usetikzlibrary{arrows,shapes,positioning}

\usepackage[all]{xy}

\newcommand{\N}{\mathbb N}
\newcommand{\Ndb}{\mathbb N}

\newcommand{\R}{\mathbb R}

\newcommand{\eps}{\varepsilon}

\begin{document}

\theoremstyle{plain}
\newtheorem{thm}{Theorem}[section]
\newtheorem{theo}[thm]{Theorem}
\newtheorem{prop}[thm]{Proposition}
\newtheorem{coro}[thm]{Corollary}
\newtheorem{lema}[thm]{Lemma}
\newtheorem{ejem}[thm]{Example}
\newtheorem{Remark}[thm]{Remark}
\newtheorem{fact}[thm]{Fact}
\newtheorem{open}[thm]{PROBLEM}

\theoremstyle{definition}
\newtheorem{defi}[thm]{Definition}
\newtheorem{rema}[thm]{Remark}

\title{Nonlinear aspects of \linebreak super weakly compact sets}

\author{G.~Lancien}
\address{Gilles Lancien, Laboratoire de Math\'ematiques de Besan\c con, Universit\'e Bourgogne Franche-Comt\'e, CNRS UMR-6623, 16 route de Gray, 25030 Besan\c con C\'edex, Besan\c con, France}
\email{gilles.lancien@univ-fcomte.fr}

\author{M.~Raja}
\address{Matias Raja, Departamento de Matem\'{a}ticas, Universidad de Murcia, Campus de Espinardo, 30100 Espinardo, Murcia, Spain}
\email{matias@um.es}

\thanks{The first named author was supported by the French ``Investissements d'Avenir'' program, project ISITE-BFC (contract
 ANR-15-IDEX-03).}
\thanks{The second named author was supported by the Grants of Ministerio de Econom\'ia, Industria y Competitividad MTM2017-83262-C2-2-P; and Fundaci\'on S\'eneca Regi\'on de Murcia 20906/PI/18.}

\keywords{super weakly compact sets, ultrapowers, uniformly convex sets, non linear embeddings in Banach spaces}
\subjclass[2010]{46B20, 46B80}

\maketitle

\begin{abstract} The notion of super weak compactness for subsets of Banach spaces is a strengthening of the weak compactness that can be described as a local version of super-reflexivity.
A recent result of K. Tu \cite{KT} which establishes
that the closed convex hull of a super weakly compact set is super weakly compact has removed the main obstacle to further development of the theory.
In this paper we provide a variety of results around super weak compactness in order to show the great scope of this notion.
We also give non linear characterizations of super weak compactness in terms of the (non) embeddability of special trees and graphs. We conclude with a few relevant examples of super weakly compact sets in non super-reflexive Banach spaces.

\end{abstract}

\section{Background}

The uniform convexity in Banach spaces and the related notion of super-reflexivity have been largely exploited along decades, as they provide a natural generalization of both finite dimensional spaces and Hilbert spaces. Some ideas behind have been distilled leading to notions such as uniformly convex function
or super-weakly compact operators. More recently, some local versions of super-reflexivity have been proposed \cite{raja, raja2, Cheng}. Before going on, let us recall that we are dealing with real Banach spaces (mostly denoted by $X$ for general results along this paper). We believe that our notation is totally standard and it can be found in fundamental books such as \cite{banach} or \cite{albiackalton} together with the basic results needed for the understanding of what follows.
We will start with the definition of super weak compactness.

\begin{defi}\label{defsuper}
A weakly closed subset $K \subset X$ is said to be super weakly compact (SWC) if  $K^{\mathcal U}$ is a relatively weakly compact subset of $X^{\mathcal U}$ for any free ultrafilter ${\mathcal U}$.
\end{defi}

Let $I$ be an infinite set and denote $\ell_\infty(X)$ the space of all bounded families $(x_i)_{i\in I}$ in $X$ equipped with the norm $\|(x_i)_{i}\|=\sup_{i\in I}\|x_i\|_X$. Given a free ultrafilter ${\mathcal U}$ on  $I$, recall that $X^{\mathcal U}$ is the quotient of $\ell_\infty(X)$ by the subspace of those $(x_i)_{i \in I}$ such that $\lim_{i, \mathcal U} \|x_i\|=0$. Then $K^{\mathcal U}$ is the set of all equivalence classes in $X^{\mathcal U}$ of families $(x_i)_{i\in I}$ such that $x_i\in K$ for all $i\in I$.
Note that in Definition \ref{defsuper}, we only ask $K^{\mathcal U}$ to be relatively weakly compact because we cannot ensure that it will be weakly closed in $X^{\mathcal U}$. For the characterization of super weak compactness, it is enough to consider just one free ultrafilter on ${\Bbb N}$ since, by the Eberlein-\v{S}mulyan theorem \cite[Theorem 3.109]{banach}, the weak compactness is separably determined (this will be further developed in section 3).

Those readers acquainted with the notion of super-reflexivity will note that Definition \ref{defsuper} implies straightforwardly that the closed unit ball $B_X$ of a Banach space $X$ is SWC if and only if $X$ is super-reflexive. However, there are examples of SWC sets which do not embed in super-reflexive spaces \cite[Example 3.11]{raja2}. Other examples of SWC sets show that they are quite ubiquitous. For instance, any weakly compact subset of $L^1(\mu)$ for $\mu$ a finite measure, or more generally of $L^1(\mu,X)$ with $X$ super-reflexive, is SWC, see \cite{raja2} and section 6 of this paper for more general results.

The definition of a super weakly compact set was introduced in \cite{Cheng} for convex sets in terms of finite representability, in a very similar way as the one for Banach spaces, see \cite{beau} or \cite{banach}. The same class of sets was previously studied by the second named author in \cite{raja} as \emph{finitely dentable sets}, which means that the dentability index $Dz(K,\varepsilon)$ is finite for every $\varepsilon>0$ (see the definition of this index in section 2). In particular, one can find there the relationship (via interpolation) with the \emph{uniformly convexifying operators} of B. Beauzamy \cite{beau1}, later called \emph{super weakly compact operators}. To explain the definition, note that an operator $T: X \rightarrow Y$ induces an operator between the ultrapowers of the spaces $T^{\mathcal U}: X^{\mathcal U} \rightarrow Y^{\mathcal U}$ for a free ultrafilter ${\mathcal U}$ on an index set as follows:  $T^{\mathcal U}((x_i)_i)=(T(x_i))_i$, for $(x_i)_i\in X^{\mathcal U}$.
Then, an operator $T: X \rightarrow Y$  is said to be super weakly compact if $T^{\mathcal U}$ is weakly compact for any ultrafilter ${\mathcal U}$ (equivalently, a free ultrafilter on ${\Bbb N}$).
A more updated account of properties of SWC convex sets (SWCC) can be found in \cite{raja2}, as well as some renorming properties of the Banach spaces generated by such sets.

The properties of non convex SWC sets have been extensively studied in the recent paper \cite{Cheng3}. Among other things, it is proved there that a set $A \subset X$ such that $A^{\mathcal U}$ is relatively weakly compact in $X^{\mathcal U}$ (such an $A$ is called {\it relatively SWC}) has SWC weak closure (Proposition 3.10 in \cite{Cheng3}). This result is quite relevant to us since the characterizations that we will provide later are actually for relative super weak compactness.

One problem left open in \cite{Cheng} was to know whether the closed convex hull of a SWC set is SWC. This has been solved in the affirmative in a recent paper by Kun Tu \cite{KT}, who provided a version of the Krein-\v{S}mulian theorem for SWC sets based on a short and clever argument. This will be a precious tool for the applications developed in this paper.

Let us describe the contents of the remaining sections of the paper.
In section \ref{RemoveConvex},  we exploit the stability of super weak compactness by closed convex hulls to derive properties for SWC sets that were known for SWCC sets. This will lead us to a characterization of super weakly compact sets as subsets of the image of a unit ball of a reflexive Banach space by a super weakly compact operator (see Theorem \ref{interpol}). This relation to SWC operators is a source of properties for SWC sets: they have the Banach-Saks property, they are uniformly Eberlein and the spaces they generate have good renormings (see \cite{raja,raja2} for SWCC sets and Corollary \ref{applicationKST} for SWC sets).

The rest of the paper is devoted to several developments around super weakly compact sets (SWC) and super weakly compact convex sets (SWCC), especially those related to non linear properties.  In  section \ref{James} we will discuss criterions to recognize super weak compactness in the absence of convexity based on a theorem of James. In section \ref{UCsets} we will introduce uniformly convex sets and discuss their properties as a tool to enjoy properties of uniformly convex norms without renorming the whole space. In  section \ref{nonlinear}, we establish analogues for SWC sets of Bourgain's \cite{bourgain} and Johnson-Schechtman's \cite{JS} metric characterizations of super-reflexivity in terms of the embeddability of dyadic hyperbolic trees, diamond graphs or Laakso graphs. These results are obtained by combining the non linear characterization of super weakly compact operators due to Causey and Dilworth \cite{CauseyDilworth} and Kun Tu's result (Theorem \ref{KST} below). We also  characterize non SWC sets in terms of the embeddabilty of the infinite dyadic tree, obtaining an analogue of Baudier's \cite{Baudier} characterization of super-reflexivity. Finally we provide in section \ref{examples} several examples and properties of SWC sets in particular Banach spaces: $L^1(\mu)$-spaces, $C(K)$-spaces, JBW$^*$-triples, $c_0$ and $L^p(\mu, X)$ spaces.

\section{First consequences of removing the convexity}\label{RemoveConvex}

Let us first us recall the most important characterizations of SWC sets among convex sets. For that
we will need some assorted definitions. % Let us remind the definition of the dentability index $Dz(K,\varepsilon)$.
Let $C$ be a bounded closed convex set of $X$. We say that $C$ is {\it dentable} if for any nonempty closed convex subset $D$ of $C$ and any
$\varepsilon >0$ it is possible to find an open halfspace $H$ of $X$ (i.e a set of the form $H=\{x\in X,\ x^*(x)>\alpha\}$, with $x^*\in X^*$ and $\alpha \in \R$) intersecting $D$ such that
$\mbox{diam}(D \cap H) \le \varepsilon$. We shall denote ${\Bbb H}$ the set of all the open half-spaces of $X$ and call ``slice of $D$'' a set of the form $D \cap H$, where $H\in {\Bbb H}$. If $C$ is dentable we may consider the following ``derivation'':
\[ [D]'_{\varepsilon} = \{ x \in D: \mbox{diam}(D \cap H)>\varepsilon,\ \text{for any}\ H \in {\Bbb H}\ \text{s.t.}\
x \in H \} .\]
Clearly, $[D]'_{\varepsilon}$
is what remains of $D$ after removing all the slices of $D$ of diameter at most
$\varepsilon$. Consider now the sequence of sets defined by $[C]_{\varepsilon}^{0}=C$ and, for every $n \in {\Bbb N}$, inductively by
\[ [C]_{\varepsilon}^{n}=[[C]_{\varepsilon}^{n-1}]'_{\varepsilon}. \]
If there is an $n$ in ${\Bbb N}$ such that $[C]_{\varepsilon}^{n-1} \not = \emptyset$ and  $[C]_{\varepsilon}^{n}  = \emptyset$ we set $Dz(C,\varepsilon)=n$. We say that $C$ is {\it finitely dentable} if $Dz(C,\varepsilon)$ is finite for every $\varepsilon>0$.

For the last section of the paper, we also need to define a {\it fragmentability index} for weakly closed and bounded subsets of $X$. It is based on a different derivation. Let $D$ be a nonempty weakly closed and bounded subset of a Banach space $X$ and $\varepsilon>0$. Our next derivation is then defined as follows:
\[ (D)'_{\varepsilon} = \{ x \in D: \mbox{diam}(D \cap V)>\varepsilon,\ \text{for any weakly open set $V$ s.t.}\
x \in V \} .\]
Similarly, $(D)'_{\varepsilon}$ is what remains of $D$ after removing all the weakly open subsets of $D$ of diameter at most $\varepsilon$. Again, for $C$ weakly closed and bounded, we define inductively  $(C)_{\varepsilon}^{0}=C$ and, for every $n \in {\Bbb N}$, $(C)_{\varepsilon}^{n}=((C)_{\varepsilon}^{n-1})'_{\varepsilon}$.
If there is an $n$ in ${\Bbb N}$ such that $(C)_{\varepsilon}^{n-1} \not = \emptyset$ and  $(C)_{\varepsilon}^{n}  = \emptyset$ we set $S(C,\varepsilon)=n$. We say that $C$ is {\it finitely fragmentable} if $S(C,\varepsilon)$ is finite for every $\varepsilon>0$.

A convex set $C \subset X$ is said to have the {\it finite tree property} if there exists $\varepsilon>0$ such that $C$ contains $\varepsilon$-separated dyadic trees of arbitrary height. Recall that a dyadic tree of height $n \in {\Bbb N}$ is a set of the form $\{ x_s: |s| \leq n \}$,
indexed by finite sequences $s \in \bigcup_{k=0}^n \{0,1\}^k$ of length $|s| \leq n$, such that
$x_s=2^{-1} ( x_{s\frown 0} +  x_{s\frown 1})$ for every $|s| < n$, where
$\{0,1\}^0:=\{\emptyset\}$ indexes the root $x_\emptyset$
and ``$\frown$'' denotes concatenation.
We say that a dyadic tree $\{ x_s: |s| \leq n \}$ is
$\varepsilon$-separated if $\| x_{s\frown 0} - x_{s\frown 1} \| \geq \varepsilon $ for every $|s|<n$.
A function $f:C \rightarrow {\Bbb R}$ defined on a convex subset $C \subset X$ is said to be {\it uniformly convex} if for every $\varepsilon>0$ there is $\delta>0$ such that $\|x-y\| < \varepsilon$ whenever $x,y \in C$ are such that
\[  \frac{f(x)+f(y)}{2} - f\Big(\frac{x+y}{2}\Big) < \delta .\]

After all these preparatory definitions the most relevant equivalences of super weak compactness for convex sets are listed in the following statement, which is taken from \cite{raja2} (Proposition 2.4).

\begin{prop}\label{superequival}
Let $X$ be a Banach space and $K \subset X$ a bounded closed convex subset. The following conditions are equivalent:
\begin{itemize}
\item[(i)] $K$ is super weakly compact;
\item[(ii)] $K$ is finitely dentable;
\item[(iii)] $K$ does not have the finite tree property;
\item[(iv)] There is a reflexive Banach space $Z$ and  a super weakly compact operator $T:Z \rightarrow X$ such that $K \subset T(B_Z)$;
\item[(v)] $K$ supports a bounded uniformly convex function;
\item[(vi)] $X$ has an equivalent norm $|\!|\!| \cdot |\!|\!|$ such that  $|\!|\!| \cdot |\!|\!|^2$ is uniformly convex on $K$.
\end{itemize}
\end{prop}

We recall that a subset $A$ of a Banach space $X$ is said to be \emph{relatively super weakly compact} (relatively SWC) if $A^{\mathcal U}$ is relatively weakly compact in $X^{\mathcal U}$.
The already mentioned result of K. Tu can be stated as follows.

\begin{theo}[\cite{KT}]\label{KST}
The closed convex hull of a relatively SWC set is SWC.
\end{theo}

As an immediate application to the only one statement from Proposition \ref{superequival} which offers no additional difficulties we obtain the following characterization in terms of interpolation.

\begin{theo}\label{interpol}
A set $K \subset X$ is super weakly compact if and only if there exists a reflexive Banach space $Z$ and a super weakly compact operator $T: Z \rightarrow X$ such that $K \subset T(B_Z)$.
\end{theo}

\begin{proof} Just use Theorem \ref{KST} together with statement $(iv)$ from Proposition \ref{superequival}.
\end{proof}

Theorem \ref{KST} allows us to remove  the difficulties of dealing only with convex sets in relation with super weak compactness. In particular, some previously known properties of SWCC sets which in their definition do not appeal to convexity are inherited by the SWC sets. Let us stress the following ones.

\begin{coro}\label{applicationKST}
Let $K \subset X$ be a SWC set. Then:
\begin{itemize}
\item[(a)] $K$ is uniformly Eberlein;
\item[(b)] $K$ has the Banach-Saks property;
\item[(c)] $K$ is finitely dentable.
\end{itemize}
None of the above implications can be reversed.
\end{coro}

\begin{proof}For SWCC, $(a)$ was established in \cite{raja} and $(b)$ in \cite[Theorem 1.3]{raja2}. Finally, remember that $(c)$ is equivalent to super weak compactness in the setting of convex sets.

There exists non super-reflexive spaces with the Banach-Saks property \cite[p. 84]{diestel}. In \cite[Example 4.9]{raja} an example of a finitely dentable weakly compact set whose closed convex hull is not finitely dentable is provided, but this example is not separable and the argument quite indirect. See Proposition \ref{exem} in this paper for a simpler example.
\end{proof}

To conclude this section, we mention that we can also remove convexity from Theorem 1.6 and Theorem 1.9 in \cite{raja2}. A Banach space $X$ is generated by a subset $K$ if the linear span of $K$ is dense in $X$. Then we will say that $X$ is \emph{super weakly compactly generated} (super WCG) if it is generated by a SWC set. A Banach space $X$ is said to be \emph{strongly generated} by a subset $K$ of $X$ if for any weakly compact subset $H$ of $X$ and any $\varepsilon>0$ there is an $n$ in ${\Bbb N}$ such that $H \subset nK+\varepsilon B_X$. Then we say that $X$ is \emph{strongly super weakly compactly generated} (S$^2$WCG) if it is strongly generated by a SWC set. It follows from Theorem \ref{KST} that these definitions coincide with the definitions given in \cite{raja2}. Let us just recall their links with renorming properties. We start with a definition. Given a bounded subset $H$ of $X$, the norm of the Banach space $(X,\| \cdot \|)$ is said to be $H$-UG smooth if
\[ \sup\{ \|x+th\|+\|x-th\|-2: x \in S_X, h \in H\} =o(|t|) \mbox{~when~}
t \rightarrow 0 . \]
The norm is said to be \emph{strongly UG smooth} if it is $H$-UG smooth for some bounded and linearly dense subset $H$  of $X$. We now reformulate the results from \cite{raja2}.

\begin{theo}\label{main_renorm}
Let $X$ be a Banach space. Then:
\begin{itemize}
\item[(a)] $X$ is generated by a SWC set (i.e. is a super WCG Banach space) if and only if it  admits an equivalent strongly UG smooth norm.
\item[(b)]
If $X$ is strongly generated by a SWC set (i.e. is a S$^{\,2}$WCG Banach space), then there is an equivalent norm on $X$ such that its restriction to any reflexive subspace of $X$ is both uniformly convex and uniformly
Fr\'echet-smooth.
\end{itemize}
\end{theo}

\section{Application of James sequences}\label{James}

For a non super-reflexive Banach space $X$, the dual of the ultrapower $X^{\mathcal U}$ is strictly larger than $(X^*)^{\mathcal U}$ and so it contains unknown or non representable elements making the study of weak compactness difficult. For that reason, the following purely intrinsic  characterization of weak compactness due to James \cite{James2}, is extremely useful for the characterization of SWC sets.

\begin{theo}[James]
A subset $C$ of $X$ is not relatively weakly compact if and only if there exists $\theta>0$ and a sequence $(x_n)_{n=0}^\infty \subset C$ such that for every $k \in {\Bbb N}$
$$ \mbox{dist}\big(\mbox{conv}\{x_j: j \leq k \}, \mbox{conv}\{x_j: j > k \} \big) \geq \theta.$$
\end{theo}

The straightforward application of the above statement to the definition of super weak compactness with ultrapowers leads to the following characterization \cite[Corollary 4.9]{Cheng3} of SWC sets, which actually provides a ``measure of non super weak compactness''.

\begin{prop}\label{SWC1}
For a closed subset $C$ of $X$ the following statements are equivalent:
\begin{itemize}
\item[(i)] $C$ is not relatively SWC.
\item[(ii)] There exists $\theta>0$ such that for every $n$ in ${\Bbb N}$ it is possible to find points $(x_k)_{k=1}^n$ in $C$ such that for every $1 \leq k < n$,
$$ \mbox{dist}\big(\mbox{conv}\{x_j: j \leq k \}, \mbox{conv}\{x_j: j > k \} \big) \geq \theta.$$
\end{itemize}
\end{prop}

Note that this can be used in the construction of arbitrary large dyadic trees inside a non SWCC set (see implication $(iii) \Rightarrow (i)$ in Proposition \ref{superequival}). Moreover, the trees obtained that way have the following  stronger separation property.

\begin{coro}
Let $K$ be closed convex non SWC. Then there exists $\delta>0$ such that for every $n$ in ${\Bbb N}$ there is a $\delta$-separated dyadic tree of height $n$ in $K$ such that the distance between nodes of consecutive levels is also at least $\delta$.
\end{coro}

\begin{proof}
Take $2^n$ points in $K$ fulfilling statement $(ii)$ of Proposition \ref{SWC1} and build a dyadic tree by averaging on dyadic partitions of $\{1,2,\dots,2^n\}$, see \cite[p. 412]{BL} for instance. Then this tree is clearly $\theta$-separated in the sense of the definition of the finite tree property. Let now $x,y$ be two nodes in consecutive levels, and assume $y$ is one level above $x$. If $x$ is not an ancestor of $y$, then both are convex combinations of the original points in such a way that $\|x-y\| \geq \theta$. If $x$ is an ancestor of $y$, then $\|x-y\| \geq \frac{\theta}{2}$. Therefore $\delta=\frac{\theta}{2}$ does the work.
\end{proof}

Then, the strong separation of trees provides non finitely dentable Lipschitz functions, thanks to an argument of Cepedello-Boiso \cite{cepe}. Recall that if $C$ is a non empty bounded subset of a Banach space $X$ and $f$ is a map from $C$ to a metric space $M$. Then $f$ is said to be \emph{finitely dentable} if for any $\eps>0$, the iteration of the derivation defined by
$$C'_\eps=\{x\in C,\ {\rm diam} f(C\cap H)>\eps\ \text{for any}\ H\in {\Bbb H}\ \text{s.t.}\ x\in H\}$$
exhausts the set $C$ in finitely many steps.

\begin{coro}\label{notappr}
Let $K$ be a closed convex non SWC set. Then there exists a non finitely dentable Lipschitz function defined on $K$.
\end{coro}

\begin{proof} As $K$ is not norm compact, it contains a uniformly separated sequence. Performing contractions at all those points we deduce that $K$ contains a uniformly separated  sequence $(K_n)$ of translations of $\lambda K$ for some $\lambda \in (0,1)$. So, there exists $\delta >0$ such that for any $n$ in $\N$, there exists a $\delta$-separated dyadic tree $T_n$ in $K_n$ with nodes of consecutive levels that are also $\delta$-separated. We may as well assume that the sequence $(K_n)$ is also  $\delta$-separated. For $n$ in $\N$, denote now $O_n$ the set of nodes at odd levels of $T_n$ and consider the closed subset $F=\bigcup_{n=1}^\infty O_n$ of $K$ and the function defined by $f(x)=d(F,x)$, which is $1$-Lipschitz. Note that any slice $S$ where the oscillation of $f$ is less than $\delta$ cannot contain points of consecutive levels of the same $T_n$. On the other hand, the derivation process cannot remove a node of the tree $T_n$ before all its descendants. Both things imply that $f$ cannot be finitely dentable.
\end{proof}

The next result adds up to the characterizations of SWCC sets  in Proposition \ref{superequival}.

\begin{theo}
A closed convex set is SWCC if and only if every real Lipschitz (or uniformly continuous) function defined on it can be uniformly approached by differences of convex Lipschitz functions.
\end{theo}

\begin{proof}
This is just a combination of \cite[Corollary 5.3]{raja} where the the approximation by differences of Lipschitz convex functions is proved for uniformly continuous functions defined on SWCC sets, and Corollary \ref{notappr} knowing that a function which is a uniform limit of differences of bounded continuous convex functions is finitely dentable, see \cite[Theorem 1.4]{raja}.
\end{proof}

The characterization of SWC sets given in Proposition \ref{SWC1} can be improved from a combinatorial point of view by reducing the number of points whose convex combinations are separated. Note
that the following characterizes non super weak compactness by the fact that ``cubes'' can be embedded uniformly in a certain fashion.

\begin{theo}\label{cube}
A subset $C$ of $X$ is not relatively super weakly compact if and only if  there exists $\theta>0$ such that for every $n$ in ${\Bbb N}$ it is possible to find a map $f_n:\{0,1\}^n \rightarrow C$ such that
$$ \mbox{dist}\big(\mbox{conv}\{f_n(A_0)\}, \mbox{conv}\{f_n(A_1)\}\big) \geq \theta   $$
for every pair of sets $A_0, A_1 \subset \{0,1\}^n$ of the form
$$A_0=(a_1,a_2,\dots,a_{k-1},0)\times \{0,1\}^{n-k}$$
$$A_1=(a_1,a_2,\dots,a_{k-1},1)\times \{0,1\}^{n-k}$$
with $1<k\leq n$ and $(a_0,\dots,a_{k-1}) \in \{0,1\}^{k-1}$.
\end{theo}

\begin{proof}
If $C$ is not relatively super weakly compact, then apply Proposition \ref{SWC1} for $2^n$ points and make the obvious arrangement. Only the reverse implication is actually an improvement, so assume now that $C$ is super weakly compact and, in order to get a contradiction, that there exists $\theta>0$ and maps $f_n:\{0,1\}^n \rightarrow C$ as in our statement. In the sequel we will denote $\{0,1\}^{\omega}$ the set of all infinite sequences in$\{0,1\}$ and $\{0,1\}^{<\omega}$, the set of all such finite sequences. If  $s$ and $t$ are sequences such that $t$ strictly extends $s$, we write $s \prec t$. For $s\in \{0,1\}^{<\omega}$, denote
$$F_s=\big\{(f_n(\sigma_{|_n}))_{n=0}^\infty,\ s\prec \sigma,\ \sigma\in \{0,1\}^\omega\big\},$$
which we view as subset of the weakly compact subset $K = \overline{C^{\mathcal U}}^{w}$ of $X^{\mathcal U}$. Note that it follows from our assumptions on the maps $f_n$ that for a fixed $n\in \N$ and for any $s\neq t \in \{0,1\}^n$, the weak closures of $F_s$ and $F_t$ have an empty intersection. Thus, we can apply Urysohn's Lemma to produce a weakly continuous map $\phi_n:K\to [0,1]$ such that
$$ \forall s\in \{0,1\}^n\ \forall x\in F_s,\ \phi_n(x)=s_n.$$
We now set $\phi=(\phi_n)_{n=1}^\infty$, which is a weakly continuous map from $K$ to $[0,1]^\N$, with the property that
$$\forall \sigma \in \{0,1\}^\omega\ \ \phi((f_n(\sigma_{|_n})_{n=0}^\infty)=\sigma.$$
Consider now the minimal weakly closed subset $H$ of $K$ such that $\phi(H)=\{0,1\}^{\Bbb N}$, whose existence is easily deduced by weak compactness. We will show that $H$ is not  $\theta$-dentable. Indeed, if $V \subset X^{\mathcal U}$ is an open half-space such that $H \cap V \not = \emptyset$ then $\phi(H \setminus V) \not = \{0,1\}^\omega$. Since $\{0,1\}^\omega \setminus \phi(H \setminus V)$ is nonempty and open in $\{0,1\}^\omega$, there exists $n\in \N$ such that $\phi_n(H \cap V)=\{0,1\}$. Since  $H$ is included in the weak closure of $F_\emptyset$, we deduce that $\phi_n(F_\emptyset \cap V)=\{0,1\}$. It follows now, again from the properties of the maps $f_n$, that there are two points in $H \cap V$ at distance not less than $\theta$. This is a contradiction because weakly compact subsets are dentable \cite[Theorem 11.11]{banach}.
\end{proof}

In the case of convex sets, James \cite[Theorem 3]{James2} proved the following more precise  characterization of weak compactness, which this time involves linear functionals.

\begin{theo}\label{james_seq}
Let $C \subset X$ be a closed convex set. Then $C$ is not weakly compact if and only if there exist $\theta>0$, and sequences $(x_n) \subset C$, $(x_n^*) \subset B_{X^*}$ such that $x^*_n(x_k)=0$ if $n > k$ and  $x^*_n(x_k)=\theta$ if
$n \leq k$.
\end{theo}

As an application, we get the following analoguous characterization of convex SWC sets.

\begin{theo}\label{james_sty}
Let $C \subset X$ be closed and convex.
The following statements are equivalent:
\begin{itemize}
\item[(i)] $C$ is not SWCC;
\item[(ii)] There exists $\theta>0$ such that for every $n \in {\Bbb N}$ there exist $(x_k)_{k=1}^n \subset C$ and $(x_k^*)_{k=1}^n \subset B_{X^*}$ such that
$x^*_n(x_k)=0$ if $n > k$ and  $x^*_n(x_k)=\theta$ if
$n \leq k$.
\end{itemize}
\end{theo}

\begin{proof} In view of Proposition \ref{SWC1}, we only need to show $(i) \Rightarrow (ii)$. So assume that $C$ is not SWCC. Then $C^{{\mathcal U}}$ is not weakly compact, so Theorem \ref{james_seq} insures the existence of $\theta>0$ and sequences $(u_n) \subset C^{{\mathcal U}}$, $(u_n^*)$ in the unit ball of $(X^{{\mathcal U}})^*$ such that $u^*_n(u_k)=0$ if $n > k$ and  $u^*_n(u_k)=\theta$ if
$n \leq k$. Now, we can use the finite representability of $C^{{\mathcal U}}$ in $C$ to find, for a fixed $n$ in $\N$, a sequence $(x_1,\ldots,x_n)$ in $C$ and a linear isomorphism $R_n$  from the linear span of $\{u_1,\ldots,u_n\}$ onto the linear span $E_n$ of $\{x_1,\ldots,x_n\}$ such that $R_n(u_i)=x_i$ for all $i\le n$, $\|R_n\|\le 2$ and $\|R_n^{-1}\|\le 2$. For $k\in \{1,\ldots,n\}$, define now a linear functional $y^*_k$ on $E_n$ by $y^*_k(x_i)=u^*_k(u_i)$. Remark that $\|y^*_k\|_{E_n^*}\le 2$ and denote $x^*_k$ its Hahn-Banach extension to $X$. Replacing $\theta$ by $\frac{\theta}{2}$, we get the desired result.
\end{proof}

For our applications to embedding results in section \ref{nonlinear}, it will be useful to refine this last result by showing that a non SWCC set $C$ is such that for any finite codimensional subspace $Y$ of $X$, $C\cap Y$ satisfies property (ii) with a parameter $\theta$ independent of $Y$. This is very much in the spirit of Corollary 5 in \cite{CauseyDilworth}, but we will detail here the version that is the most adapted for our embedding questions. As usual, we start with a statement about non weakly compact sets.

\begin{prop}\label{Jamescodim} Let $C \subset X$ be closed, convex, bounded, symmetric and not relatively weakly compact. Then, there exists $\theta >0$ such that for every finite codimensional subspace $Y$ of $X$, there exists $x^{**} \in \overline{C\cap Y}^{w^*}$, the weak$^*$-closure of $C\cap Y$ in $X^{**}$, such that $d(x^{**},X)>\theta$.
\end{prop}

\begin{proof} First, we claim that for every finite codimensional subspace $Y$ of $X$, there is a compact subset $K$ of $X$ such that $C$ is included in $(3C\cap Y)+K$. Indeed, denote by $Z$ the space linearly spanned by $C$ and equipped with $|\ |_C$, the Minkowski functional of $C$, as the norm. It is standard that $(Z,|\ |_C)$ is a Banach space (see \cite[Exercise 2.22]{banach} for instance). Since $C$ is bounded, the identity mapping $I$, from $Z$ to $X$ is bounded. Then $W=I^{-1}(Y)$ is also a closed finite codimensional subspace of $Z$. Then, a well known application of the Bartle-Graves selection theorem (see \cite{DGZ}) insures that there exists a compact subset $L$ of $Z$ such that $B_Z \subset 3B_W+L$. Applying the map $I$ and denoting $K=I(L)$ finishes the proof of our claim.

Assume now that the conclusion of the proposition is false. Then for any $\eps>0$, there exists a finite codimensional subspace $Y$ of $X$ such that $\overline{C\cap Y}^{w^*} \subset X+\eps B_X$. But our first claim implies that $\overline{C}^{w^*}\subset 3 \overline{C\cap Y}^{w^*}+K$, for some compact subset $K$ of $X$. We deduce that for any $\eps>0$, $\overline{C}^{w^*}\subset X+3\eps B_X$, which implies that $C$ is weakly compact, a contradiction.
\end{proof}

We now turn again to non super weakly compact sets.

\begin{theo}\label{SWCcodim} Let $C \subset X$ be closed, convex, bounded and symmetric. The following statements are equivalent.
\begin{itemize}
\item[(i)] $C$ is not SWCC;
\item[(ii)] There exists $\theta>0$ such that for every  finite codimensional subspace $Y$ of $X$, every $n \in \N$, there exist $(x_k)_{k=1}^n \subset C\cap Y$ and $(x_k^*)_{k=1}^n \subset B_{X^*}$ such that
$x^*_n(x_k)=0$ if $n > k$ and  $x^*_n(x_k)=\theta$ if
$n \leq k$.
\end{itemize}
\end{theo}
\begin{proof} Again, we only need to show $(i) \Rightarrow (ii)$. So assume that $C$ is not SWCC, let $\mathcal U$ be a non trivial ultrafilter on $\N$ and $Y$ be a finite codimensional subspace of $X$. Then, using norm compactness in a finite dimensional complement of $Y$ in $X$, we easily see that $Y^{\mathcal U}$ is a finite codimensional subspace of $X^{\mathcal U}$. Then, identifying $(C\cap Y)^{\mathcal U}$ with $C^{\mathcal U}\cap Y^{\mathcal U}$ and using Proposition \ref{Jamescodim} we deduce the existence of $\theta>0$, independent of $Y$, such that  $\overline{(C\cap Y)^{\mathcal U}}^{w^*}$ (the weak$^*$ closure is meant here in $(X^{\mathcal U})^{**}$) has points at distance from $X^{\mathcal U}$ greater than $\theta$. Now, the proof of James' theorem (Theorem \ref{james_seq} of this paper) provides us with
sequences $(u_n)_{n=1}^\infty$ in $(C\cap Y)^{\mathcal U}$, $(u_n^*)_{n=1}^\infty$ in the unit ball of $(X^{\mathcal U})^*$ (actually in the unit ball of $(Y^{\mathcal U})^*$, but we may consider their Hahn-Banach extensions) such that $u^*_n(u_k)=0$ if $n > k$ and  $u^*_n(u_k)=\theta$ if
$n \leq k$. We conclude, similarly to the proof of Theorem \ref{james_sty}, by using the finite representability of $(C\cap Y)^{\mathcal U}$ in $C\cap Y$.
\end{proof}

\section{Uniformly convex sets}\label{UCsets}

Let us say that a symmetric bounded closed convex set $K$ is {\it uniformly convex} if for every $\varepsilon>0$ there is $\delta>0$ such that
$$\forall x,y \in K,\ \ \|x-y\|>\varepsilon\ \Rightarrow \frac{x+y}{2} \in (1-\delta)K.$$
There is a more popular definition of uniform convexity for sets in finite dimension, but its natural extension to Banach spaces \cite{Polyak} is not equivalent to ours and only super-reflexive spaces can contain such sets. For a uniformly convex set $K$ we may define the convexity modulus as
$$ \delta_K(\varepsilon) = \inf\Big\{ 1 - \big|\frac{x+y}{2}\big|_K: x,y \in K, \|x-y\| \geq \varepsilon\Big\},$$
where $| \cdot |_K$ is the Minkowski functional of $K$. Note that if $K$ is not a segment, then there is a 2-dimensional subspace $Y$ such that $K \cap Y$ is an equivalent uniformly convex ball on $Y$, therefore $\delta_K(\varepsilon) \leq c \varepsilon^2$, for some $c>0$. On the other hand, a lower bound for $\delta$ of power type is not guaranteed as we will see later in examples.

\begin{lema}\label{triang}
Let $K$ be uniformly convex and $| \cdot |_K$ its Minkowski functional. Then whenever $x,y \in K$ we have
$$ \big| \frac{x+y}{2} \big|_K \leq \max\{ |x|_K,|y|_K\} - \delta_K(\|x-y\|). $$
\end{lema}

\begin{proof} Note that we always have $\delta_K(2\|x\|) \leq |x|_K$. Thus we may assume $x \not = -y$. Let
$$\lambda =1- \max\{|x|_K,|y|_K\}\ \text{and}\  z=\frac{x+y}{|x+y|_K}.$$
Observe that $x+\lambda z, y+\lambda z \in K$ and so
$$ \big| \frac{x+y}{2} +\lambda z \big|_K = \big| \frac{x+y}{2} \big|_K + \lambda \leq 1 - \delta_K(\|x-y\|) $$
which implies the desired inequality.
\end{proof}

\begin{prop}
Any uniformly convex set is SWC. Reciprocally, any SWC set is contained in some uniformly convex set.
\end{prop}

\begin{proof} Uniform convexity implies that $[K]'_\varepsilon \subset (1-\delta_K(\varepsilon)) K$. Then an homogeneity argument clearly yields that $[K]^n_\varepsilon \subset (1-\delta_K(\varepsilon))^n K$. So, there exists $n$ in $\N$ such that the $\|\ \|$-diameter of $[K]^n_\varepsilon$ is smaller than $\varepsilon$ and therefore such that $[K]^{n+1}_\varepsilon$ is empty. On the other hand, if $K$ is SWC, by Theorem \ref{interpol} and \cite{beau1}, there is a uniformly convex operator $T: Z \rightarrow X$ such that $K \subset T(B_Z)$. It is then obvious that $\overline{T(B_Z)}$ is uniformly convex.
\end{proof}

\begin{prop}
Let $K$ be SWCC and symmetric. Then for every $\delta>0$ there is a uniformly convex set $C$ such that $ K \subset C \subset (1+\delta)K$.
\end{prop}

\begin{proof} If $Z$ is the linear span of $K$ and $|\cdot |_K$ the Minkowski functional of $K$, then $(Z,|\ |_K)$ is a Banach space (see again \cite[Exercise 2.22]{banach}) and the inclusion of $Z$ into $X$ is a super weakly compact operator. Therefore, there is a renorming $| \cdot |_u$ of $Z$ making this operator uniformly convex. Note that all the norms $| \cdot |_K + \varepsilon |\cdot |_u$ for $\varepsilon>0$ make the operator uniformly convex. Then, for $\varepsilon>0$ small enough, the unit ball $C$ of $| \cdot |_K + \varepsilon |\cdot |_u$ provides the desired set.
\end{proof}

Finally we will prove an intrinsic version of Kadec's theorem \cite{Kadec}.

\begin{prop}
Let $K$ be a uniformly convex set with modulus of convexity $\delta_K$. Then for any finite or infinite sequence $(x_n)$ such that $\sum_n \eps_n x_n \in K$ for any $(\eps_n)_n$ in $\{-1,1\}$, we have that $\sum_n \delta_K(2 \|x_n\|) \leq 1$.
\end{prop}

\begin{proof}
As above $|\cdot |_K$ stands for the Minkowski functional of $K$.
Modifying the signs we may suppose without loss of generality that
$$ | x_1 +\dots+x_{k-1}-x_k |_K \leq | x_1 +\dots+x_{k-1}+x_k  |_K .$$
By Lemma \ref{triang} we have that for all $k>1$:
$$ \delta_K(2\|x_k\|) \leq   | x_1 +\dots+x_{k-1}+x_k  |_K -  | x_1 +\dots+x_{k-1} |_K.$$
Since we also have that $\delta_K(2\|x_1\|)\le |x_1|_K$, summing up, we obtain
$$ \sum_{n=1}^\infty \delta_K(2\|x_n\|) \leq  | x_1 +\dots+x_{k-1}+x_k  |_K  \leq 1 .$$
\end{proof}

\section{Relation to metric trees, diamonds graphs and Laakso graphs}\label{nonlinear}

The characterization of the non super-reflexivity of a Banach space $X$ by the equi bi-Lipschitz embeddability of the family of binary trees of arbitrary height equipped with the hyperbolic metric, due to J. Bourgain \cite{bourgain}, is one of the milestones of the non linear geometry of Banach spaces. F. Baudier \cite{Baudier} completed this result by showing that it is also equivalent to the Lipschitz embeddability of the infinite binary tree. Then, Johnson and Schechtman \cite{JS} showed that super-reflexivity is also characterized by the non equi bi-Lipschitz embeddability of the diamond graphs or of the Laakso graphs (we refer the reader to \cite{JS} for their precise definitions). In this section we describe the analogous characterization of (non) relative super weak compactness.

In order to illustrate this section, we will only recall the definition of the simplest of these families: the metric binary trees. For $N\in \N$, we denote $T_N=\{\emptyset\}\cup \bigcup_{n=1}^N \{0,1\}^n$. There is a natural order on $T_N$ defined by $s\preceq t$ if the sequence $t$ extends $s$. This allows us to introduce the greatest ancestor of $s$ and $t$ denoted $a_{s,t}$. For $s\in T_N$, we denote $|s|$ the length of $s$. We now define a distance on $T_N$ by the formula
\begin{equation}\label{treemetric}
 d(s,t)=d(a_{s,t},s)+d(a_{s,t},y)=|s|+|t|-2|a_{s,t}|.
\end{equation}
The most natural way to describe this distance is as the graph (or shortest path) metric of $T_N$ equipped with its natural graph structure (two sequences are adjacent if one of them is the immediate predecessor of the other in the ordering $\preceq$). For $s\in T_N$, $s^+$ denotes the set made of the two immediate successors of $s$ for $\preceq$.

We also need to add some notation and terminology. Let $f: (M,d) \rightarrow (X,\|\ \|)$ be a map from a metric space into a Banach space. The average range of $f$ is the following subset of $X$
$$ \mbox{ave}(f) = \left\{ \frac{f(s)-f(t)}{d(s,t)}: s,t \in M, s \not =t \right\} .$$
Obviously a map is Lipschitz if $\mbox{ave}(f)$ is bounded.
We say that $f$ is \emph{$\theta$-separated} if $\|f(s)-f(t)\| \geq \theta \, d(s,t)$, which is equivalent to say that the inverse map is $\theta^{-1}$-Lipschitz. A family of maps from $M$ into $X$ is said \emph{uniformly separated} if they are all  $\theta$-separated for some $\theta>0$.

We can now state our result.

\begin{theo}\label{embeddings} Let $L$ be a bounded subset of a Banach space $(X,\|\ \|)$. Let $(M_N,d_N)_{N\in \Ndb}$ be any one of the following families : binary trees, diamond graphs, Laakso graphs. Then $L$ is not relatively SWC if and only if there exist uniformly separated embeddings $f_N: (M_N,d_N) \rightarrow X$ such that  $\mbox{ave}(f_N)$ is included in $K=\overline{\mbox{aco}}(L)$, the closed absolute convex hull of $L$.
\end{theo}

\begin{proof} We know from Theorem \ref{KST} that $L$ is not relatively SWC if and only if $K$ is not SWC. Denote also $Z$ the linear span of $K$ and $T$ the identity from $(Z,|\ |_K)$ to $(X,\|\ \|)$. We may as well assume that $L\subset B_X$ and thus that $\|T\|\le 1$. Note now that $K$ is not SWC if and only if the operator $T$ is not SWC. Although it is not exactly stated in these terms, it follows from the work of Causey and Dilworth \cite{CauseyDilworth} that $T$ is not SWC if and only if there exists $\theta>0$ and maps $f_N: M_N \to Z$ such that
\begin{enumerate}[(i)]
\item For all $s,t\in M_N$, $|f_N(s)-f_N(t)|_K\le d_N(s,t)$.
\item For all $s,t\in M_N$, $\|f_N(s)-f_N(t)\|\ge \theta d_N(s,t)$.
\end{enumerate}
This concludes the proof.
\end{proof}

\begin{rema} The above results apply to unit balls of Banach spaces, which allows to recover Bourgain's theorem. Note that in this generalization it is very important that the characterization is given in terms of the norm of the ambient space for the separation and in terms of the Minkowski functional of $K$ for the Lipschitz constant.
\end{rema}

\begin{rema} For a very general approach, we refer the reader to the recent paper by A. Swift \cite{Swift}, where it is shown (Theorem 6.7) that super-reflexivity is equivalent to the non equi-Lipschitz embeddability of any family of bundle graphs generated by a given finitely branching bundle graph.
\end{rema}

The article by Causey and Dilworth \cite{CauseyDilworth} is written in terms of super weakly compact operators and applies to symmetric convex sets. We include below a proof, for the case of trees, using only the tools of our paper.

\begin{proof}[Proof of Theorem \ref{embeddings}] Assume that $L$ is not relatively SWC. Then $K$ is not SWC and there exists $\theta>0$ such that for every $N \in {\Bbb N}$ there exists $(x_k)_{k=1}^{2^{N+1}-1}$ in $K$ and $(x^*_k)_{k=1}^{2^{N+1}-1}$ in $B_{X^*}$ satisfying condition (ii) in Theorem \ref{james_sty}. Bourgain's map \cite{bourgain}, see also \cite[Lemma 13.11]{Pisier}, defined by
$$ f(s) = \sum_{t \preceq s} x_{\sigma(t)} ,$$
where $\sigma: T_N \rightarrow \{1,\dots,2^{N+1}-1\}$ is a suitable labelling of the nodes of the tree (see \cite{bourgain} or \cite{Baudier} for details), is clearly $\theta$-separated and its average range is in $K$.

Assume now that $L$ is relatively SWC, and thus, by Theorem \ref{KST} that $K$ is SWC. Aiming for a contradiction, assume also that there exists $f:T_N\to X$ which is $\theta$- separated with $\mbox{ave}(f) \subset K$. Since $K$ is SWC, we may assume, without loss of generality that $K$ is uniformly convex with modulus $\delta$. Let $|\ |_K$ be the Minkowski functional of $K$ and notice that $f$ is $C$-Lipschitz for some $C \leq 1$ if we endow $K$ with $|\ |_K$.
Now we will show that Kloeckner's fork argument \cite{Kloeckner} is valid in this context. Given nodes $s_0,s_1,s_2,s_2'$ such that $s_1 \in s_0^+$ and $\{s_2,s_2'\} = s_1^+$, we claim that
$$ \min\{ |f(s_0) - f(s_2)|_K, |f(s_0) - f(s_2')|_K \} \leq 2(C-\delta(\theta)).$$
Indeed, assume not and set $x = f(s_0) - f(s_1)$, $y=f(s_1) - f(s_2)$ and
$y'=f(s_1) - f(s_2')$ which all are in $K$ and so that  $|x|_K,|y|_K,|y'|_K\leq C$. Then we have
$$\big|\frac{x+y}{2}\big|_K > C-\delta(\theta)\ \  \text{and}\ \  \big|\frac{x+y'}{2}\big|_K  > C-\delta(\theta),$$
which imply $\|x-y\|<\theta$ and $\|x-y'\|<\theta$, and therefore
$$ \|f(s_2) - f(s_2')\| = \|y-y'\|<2\theta$$
contradicting the $\theta$-separation and proving our claim.\\
If $N$ was even, the application of this argument would provide a selection of nodes equivalent to $T_{\frac{N}{2}}$, on which the restriction of $f$ is $\theta$-separated with respect to $\| \ \|$ and the Lipschitz constant has been reduced to $C-\delta(\theta)$ with respect to $|\ |_K$. Starting with a tree of height $N=2^{k+1}$, the recursive application of this argument would provide a $\theta$-separated map from $T_2$ with Lipschitz constant $C-k\delta(\theta)$, which is impossible for large values of $k$.
\end{proof}

We conclude this section by showing a metrical characterization of super weak compactness that is the exact analogue of Baudier's characterization of super-reflexive Banach spaces \cite{Baudier}. Let us denote $T_\infty$ the union of all the $T_N$'s for $N\in \N$, that we equip with the distance $d$ defined as in (\ref{treemetric}).

\begin{theo}\label{Baudier} Let $L$ be a bounded subset of a Banach space $(X,\|\ \|)$. Then $L$ is not relatively SWC if and only if there exist $\theta>0$ and a $\theta$-separated map  $f: (T_\infty,d) \rightarrow X$ such that  $\mbox{ave}(f)$ is included in $K=\overline{\mbox{aco}}(L)$, the closed absolute convex hull of $L$.
\end{theo}

\begin{proof} It follows from Theorem \ref{embeddings} that we only need to show one implication. So assume that $L$ is not relatively SWC and let us build an embedding $f$. Armed with Theorem \ref{SWCcodim}, we only need to reproduce Baudier's original barycentric gluing argument \cite{Baudier}. So we will just recall the main steps of his construction.

Denote $k_n$ the cardinality of $T_{2^{n+1}}$. Then we can build inductively, using Theorem \ref{SWCcodim} and a standard Mazur gliding hump argument, subspaces $(F_n)_{n=1}^\infty$ of $X$, points $x_{n,1},\ldots,x_{n,k_n}$ in $K\cap F_n$, linear functionals $x_{n,1}^*,\ldots,x_{n,k_n}^*$ in $B_{X^*}$ so that $x_{n,k}^*(x_{n,i})=0$ if $k>i$ and $x_{n,k}^*(x_{n,i})=\theta$ if $k\le i$, for some fixed $\theta >0$. We also make sure in the construction that $(F_n)_{n=1}^\infty$ is a Schauder decomposition of its closed linear span $Z$. Let now $\sigma_n:T_{2^{n+1}}\to \{1,\ldots,k_n\}$ be an enumeration of $T_{2^{n+1}}$ following the lexicographic order. Now define
$$f_n(\emptyset)=0\ \text{\ and\ }\forall s \in T_{2^{n+1}}\setminus \{\emptyset\},\ \ f_n(s) = \sum_{t \preceq s} x_{n,\sigma_n(t)}.$$
Finally, still following Baudier's lead, we define $f:T_\infty \to Z\subset X$ as follows: $f(\emptyset)=0$ and if $2^n\le |s| \le 2^{n+1}$, for some $n\in \N\cup\{0\}$, then
$$f(s)=\lambda f_n(s)+(1-\lambda)f_{n+1}(s),\ \text{ where }\ \lambda=\frac{2^{n+1}-|s|}{2^n}.$$
We have now gathered all the ingredients to follow the estimates carried out in \cite{Baudier} and conclude that there exists $C,\eta >0$ such that $f$ is $C$-Lipschitz as a function with values in the linear span of $K$ equipped with $|\ |_K$ and $\eta$-separated. A final rescaling of $f$ by a factor $\frac1C$ yields the conclusion.
\end{proof}

\section{Examples of super weakly compact sets}\label{examples}

Here we will discuss some examples that we believe to be interesting.

\begin{prop}\label{lone}
For any measure space $(\Omega, \Sigma, \mu)$ and for any compact Hausdorff space $K$ all the weakly compact subsets of $L^1(\Omega,\mu)$ and $C(K)^*$ are SWC. Therefore, any weakly compact operator with range $L^1(\mu)$ or domain $C(K)$ is super weakly compact.
\end{prop}

\begin{proof} We may proceed only with separable weakly compact subsets of $L^1(\Omega, \mu)$ as super weak compactness is separably determined. Note that any separable subset of $L^1(\Omega, \mu)$ is supported on a $\sigma$-finite set $\Omega'$.
On the other hand $L^1(\Omega',\mu)$ is isometric to some space  $L^1(\nu)$ where $\nu$ is a finite measure. As $L^1(\nu)$ is strongly generated by the unit ball of $L^2(\nu)$, we deduce that  any weakly compact subset of $L^1(\Omega, \mu)$ is SWC, see \cite[p. 123]{beau1} or \cite[Proposition 2.7 (5)]{raja2}. The proof for $C(K)^*$ is similar. Note that this space can be decomposed as an $\ell_1$ sum of $L^1(\mu)$ spaces.

The consequence for weakly compact operators with range in an $L^1(\mu)$ space is evident. Let now $T: C(K) \rightarrow X$ be a weakly compact operator. Then $T^*: X^* \rightarrow C(K)^*$ is weakly compact too, and thus it is super weakly compact. Then $T$ is again super weakly compact by Beauzamy's duality.
\end{proof}

The next statement then follows straightforwardly from the application of a well known result of Pe\l cy\'nski, see \cite[Corollary 5.6.4]{albiackalton} for instance.

\begin{coro}
Every operator from a $C(K)$ space into a Banach space which contains no copy of $c_0$ is super weakly compact.
\end{coro}

Note that the $C(K)$ case in Proposition \ref{lone} includes in particular $L^\infty(\mu)$ spaces, see \cite[Lemma 5.3]{rodriguez} for a result is this direction. The particular properties of $L^1(\mu)$ as a Banach lattice may suggest a possible generalization of Proposition \ref{lone} in this setting. Actually, we propose here an upgrade of the above result that is rather based on its algebraic structure. We shall deal with preduals of JBW$^*$-triples, which include in particular preduals of Von Neumann algebras and thus the complex $L^1(\mu)$  spaces. We are grateful to Ond\v{r}ej F.K. Kalenda who kindly provided us with the following result and the arguments for the proof below. The definition of a JBW$^*$-triple and basic related information can be found for instance in the papers \cite{BHKPP,HKPP}.

\begin{theo}\label{JBW}
Every weakly compact subset of a JBW$^*$-triple predual is SWC.
\end{theo}

\begin{proof}
Let us start by noticing that any JBW*-triple $E$ has a unique predual  \cite{BarTi}. As in the proof of the previous result we may suppose that the weakly compact set $K \subset E$ is separable. By \cite[Theorem 1.1]{BHKPP} $E$ is 1-Plichko. In particular, it has 1-SCP \cite[Corollary 1.3]{BHKPP}, that is, every separable subspace is contained in a $1$-complemented separable subspace, so we may assume that $K \subset F$ where $F$ is $1$-complemented in $E$. Now, we claim that $F^*$ is a JBW$^*$-triple. Indeed, by \cite{Kaup} a 1-complemented subspace of a JB$^*$-triple is again a JB$^*$-triple and the claim follows by duality.
Since $F$ is a separable predual of a JBW$^*$-triple, it is strongly WCG \cite[Corollary 9.4]{HKPP}. It remains to show that $F$ is actually strongly generated by a SWC set which would imply that $K$ is SWC by \cite[Proposition 2.7 (5)]{raja2}. In order to do that it is necessary to look into the extra information provided by the proofs in \cite{HKPP}.
The compact $K(\phi)$ that strongly generates $F$ in \cite[Theorem 9.3(c)]{HKPP} (see also \cite[Proposition 7.11(b)]{HKPP}) comes from
a Hilbert space. Indeed, $K(\phi)$  is defined in \cite[Lemma 7.10(b)]{HKPP} and it follows from
the formula for $\Phi$ in \cite[Lemma 7.10(a)]{HKPP} that it factors through a Hilbert space, which is the completion of $M$ endowed with the inner
product $[x,y]=\phi(\{x,y,e\})$, see \cite[Proposition 3.2]{BHKPP} and its proof for additional details.
\end{proof}

We can then deduce, and extend, the following classical result of W. Szlenk \cite{Szlenk}.

\begin{coro}
A weakly convergent sequence in $L^1(\mu)$, or more generally in a JBW$^*$-triple predual, has a subsequence whose Ces\`{a}ro mean converges in norm to the same limit.
\end{coro}

 \begin{proof} A weakly convergent sequence together with its limit is a weakly compact set, which in $L^1(\mu)$ or a JBW$^*$-triple predual is SWC and therefore, by Corollary \ref{applicationKST} has the Banach-Saks property.
 \end{proof}

Now we will consider subsets of $c_0$ which are families of  characteristic functions of finite sets of
${\Bbb N}$, namely sets of the form  $K=\{ \chi_F: F \in {\mathcal F}\}$ with ${\mathcal F} \subset [ \Bbb N ]^{<\omega}$. Here $[ \Bbb N ]^{<\omega}$ denotes the set of finite subsets of $\N$. For a finite set $F$, we denote $|F|$ its cardinality.  Observe that for every $p \in {\Bbb N}$ the family ${\mathcal F}= \{F \subset {\Bbb N}:  |F| \leq p \}$ produces a SWC set as it is covered by $I(p B_{\ell^2})$ where $I:\ell_2 \rightarrow c_0$ is the canonical injection, which is a SWC operator. We now give a necessary condition for such a set $K$ to be SWC in $c_0$.

\begin{prop}\label{c_0subsets}
Let ${\mathcal F} \subset [ \Bbb N ]^{<\omega}$ a family of subsets such that $K=\{ \chi_F: F \in {\mathcal F}\}$ is a SWC subset of $c_0$. Then, there exists $p\in \N$ and $C>0$ such that
$$\forall A \in [\Bbb N]^{<\omega},\ \  \big|\{ F \cap A: F \in {\mathcal F} \}\big| \le C|A|^p.$$
\end{prop}

\begin{proof} Actually, we will show that the cardinality of the set in the statement is below $N^p$ for some $p \in {\Bbb N}$ and $N= |A|$ large enough. Assume that the result is false. So for any fixed $p \in {\Bbb N}$
the cardinality is not eventually bounded by $N^p$. Note that the expression
$$ {N \choose 0} + {N \choose 1}  + \dots + {N \choose p-1} $$
is a polynomial of degree $p-1$ on $N$, so by our assumption there is $N \in \N$ such that
$$ \big|\{ F \cap A: F \in {\mathcal F} \}\big| >  {N \choose 0} + \dots + {N \choose p-1} .  $$
The Sauer-Shelah Lemma then insures that there is a subset $S$ of $A$ with $|S|=p$ such that
$$ \{ F \cap S: F \in {\mathcal F} \} = 2^S. $$
Now, take an enumeration of $S=\{n_1,\dots, n_p\}$ and define points $(x_k)_{k=1}^p \subset K$ of the form $x_k=\chi_{F_k}$ where $F_k \in {\mathcal F}$ is such that $F_k \cap S=\{n_1,\dots, n_k\}$. In particular, $n_k\in F_k$ if and only if $i\ge k$, which clearly implies that $(x_k)_{k=1}^p$ satisfies
Proposition \ref{SWC1} whith $\theta =1$. As $p$ can be arbitrarily large, $K$ is not SWC.
\end{proof}

Another non trivial example leading to an explicit estimate is provided by the Schreier family
${\mathcal S}$, which is made up of those $F \in [{\Bbb N}]^{\omega}$ such that $|F| \leq \min(F)$,
we have that $\{\chi_F: F \in {\mathcal S}\}$ is weakly compact in $c_0$ while
$$ \big|\{ F \cap [1,N]: F \in {\mathcal S} \}\big| > 2^{\frac{N}{2}-1}.$$

However, the property of Proposition \ref{c_0subsets} does not characterize super weak compactness in $c_0$, as it is shown by the following example.

\begin{prop}\label{exem}
There exists a family ${\mathcal F}$ of finite subsets of $\N$ such that $K=\{ \chi_F: F \in {\mathcal F}\}$ is weakly compact, finitely dentable, non SWC and so that for every finite set $A \subset {\Bbb N}$ we have
$$ \big|\{ F \cap A: F \in {\mathcal F} \}\big| =  |A| +1.$$
\end{prop}

\begin{proof} Define for every $n \in {\Bbb N}$ and $m \in\{1,\ldots,n\}$ the sets
$$ F_{n,m} = \Big\{ k \in {\Bbb N}: \frac{n(n-1)}{2} < k \leq \frac{n(n-1)}{2}+ m \Big\}.$$
Consider the families ${\mathcal F}_n = \{F_{n,m}: 1\le m \leq n\}$ and ${\mathcal F} = \{\emptyset\} \cup \bigcup_{n \in {\Bbb N}} {\mathcal F}_n$. It is easy to see that any sequence in $K=\{ \chi_F: F \in {\mathcal F}\}$ admits a subsequence that is either stationary or weakly null. Therefore, $K$ is weakly compact. Note that for every $n$ in ${\Bbb N}$ the sequence $(x_m)_{m=1}^n = (\chi_{F_{n,m}})_{m=1}^n$ satisfies condition (ii) of Proposition \ref{SWC1} with $\theta=1$, and so $K$ is not SWC. Now we will show that any $x \in K \setminus \{0\}$ can be separated with a slice from the rest of $K$. If $x=\chi_{F_{n,m}}$ take $x^*=(a_k)_{k=1}^\infty \in \ell_1$, where $a_k$ is
$$1\ \text{if}\ k \in F_{n,m};\ \ -1\ \text{if}\ \frac{n(n-1)}{2} + m < k \leq \frac{n(n+1)}{2};\ \ 0\ \text{otherwise}.$$
With this choice, we have that
$\{x\} = \{y \in K: x^*(y) > m -1/2\}$, which shows that $[K]'_\varepsilon=\{0\}$ for any $\varepsilon <1$ and  implies the finite dentability of $K$. Finally, observe that
$$\big|\{ F \cap A: F \in {\mathcal F}_n, \ F \cap A\neq \emptyset\}\big| = \Big|A \cap \big(\frac{n(n-1)}{2},\frac{n(n+1)}{2}\big]\Big|$$ leading to the estimation of the statement (adding $1$ for $F=\emptyset$).
\end{proof}

\begin{rema} The set $K$ we just described is isometric to $\N$, equipped with the discrete metric, and the same is true for the subset $L$ of $c_0$ made of $0$ together with the elements of the  canonical basis of $c_0$. However, $K$ is not SWC, while $L$ is SWC (for instance, because it is weakly compact and included in the image of the unit ball of $\ell_2$ by the identity map from $\ell_2$ in $c_0$). This underlines the fact that convexity or the use of convex hulls seems unavoidable if one is looking for metrical characterizations of  super weak compactness.
\end{rema}

Now we will discuss super weak compactness of some subsets of the Bochner-Lebesgue space $L^p(X)=L^p([0,1], X)$. Our starting point is the following result of Beauzamy \cite[Proposition II.3]{beau1} restated in terms of super weak compactness.

\begin{theo}[Beauzamy]\label{Beauzamy} Let $X,Y$ be two Banach spaces and $p\in (1,+\infty)$. If $T:X \rightarrow Y$ is super weakly compact, then the induced operator $T_p:L^p(X) \rightarrow L^p(Y)$, defined by $T_pf(t)=T(f(t))$, is super weakly compact.
\end{theo}

For a subset $K$ of $X$, we denote $L^p(K)$ the set of all $f$ in $L^p(X)$ whose essential range is included in $K$. Then, we deduce the following.

\begin{coro}
A subset $K$ of $X$ is SWC if and only if $L^p(K)$ is SWC for some or all $p$ in $(1,+\infty)$.
\end{coro}

\begin{proof} Note that $K$ is linearly isometric to a closed subset of $L^p(K)$. So we have only one implication to show and we assume that $K$ is SWC. So, there exists a Banach space $Z$ and a super weakly compact operator $T: Z\to X$ such that $K \subset T(B_Z)$. Then $L^p(K) \subset T_p(B_{L^p(Z)})$ is super weakly compact by Theorem \ref{Beauzamy}.
\end{proof}

This last result leads to the following characterization of SWC sets by means of the fragmentability index.

\begin{theo}
A weakly compact subset $K \subset X$ is SWC if and only if $L^2(K)$ is finitely fragmentable.
\end{theo}

\begin{proof} Let $H$ denote the closed convex hull of $K$. Assume first that $K$ is SWC, then so is $H$ and then $L^2(H)$ is SWC. Indeed, property (iv) in Proposition \ref{superequival} insures the existence of a super weakly compact operator $T:Z\to X$ with $Z$ a reflexive Banach and such that $H\subset T(B_Z)$. Then, we obtain that $L^2(H)$ is a subset of $T_2(B_{L^2(Z)})$, with $L^2(Z)$ being reflexive and $T_2$ super weakly compact by Theorem  \ref{Beauzamy}. Thus $L^2(H)$ is SWC and hence finitely dentable. Note now that $L^2(K) \subset L^2(H)$ and that its fragmentability index is bounded by its dentability index and therefore is finite.

Assume now that $L^2(K)$ is finitely fragmentable. Firstly we will reduce the problem to the convex case by showing that $L^2(H)$ is finitely fragmentable. In order to do that, consider the bounded operator $T: L^2([0,1]^2,X) \rightarrow L^2([0,1],X)$ defined by
$$ T(f)(t) = \int_0^1 f(t,s) \, ds .$$
Let us remark that there is an isometry  $U$ from $L^2([0,1]^2,X)$ onto $L^2([0,1],X)$ such that $U(L^2([0,1]^2,K))=L^2([0,1],K)$. Now note that any simple function with values in $\mbox{conv}(K)$ can be uniformly approximated by elements from $T(L^2([0,1]^2,K))$. That implies that $L^2(H) \subset T(L^2([0,1]^2,K)) $ and thus the desired result, as the fragmentability index of the linear continuous image remains finite. Now we can apply an argument of the first named author \cite{LancienT, Lancien3} that states
$$ Dz(H, \varepsilon) \leq S\big(L^2(H), \frac{\varepsilon}{2}\big).$$
Therefore $H$ is finitely dentable and hence SWC.
\end{proof}

\noindent{\bf Acknowledgements.} This work was initiated while the first named author was visiting the Universidad de Murcia and completed while the second named author was a visiting professor at the Universit\'e de Franche-Comt\'e. Both authors wish to thank these institutions for their support. The authors also thank O. Kalenda for providing us with Theorem \ref{JBW} and G. Grelier for fruitful discussions.

\end{document}